\documentclass[reqno,11pt]{amsart}
\usepackage{amsmath,amssymb,amsfonts,amsthm, amscd,indentfirst}
\usepackage{amsmath,latexsym,amssymb,amsmath,
	amscd,amsthm,amsxtra}

\usepackage[hyphens]{url}
\usepackage[colorlinks=true,linkcolor=blue,citecolor=blue,urlcolor=blue,hypertexnames=false,linktocpage]{hyperref}
\usepackage{bookmark}
\usepackage{amsmath,thmtools,mathtools}
\mathtoolsset{showonlyrefs=true}
\usepackage{enumerate}

\usepackage[T1]{fontenc}
\usepackage{newtxtext,newtxmath}

\newcounter{mparcnt}

\usepackage{fancyhdr}
\usepackage{esint}
\usepackage{enumerate}

\usepackage{pictexwd,dcpic}
\usepackage{graphicx}
\usepackage{caption}
\usepackage{xcolor}
\usepackage{graphicx}
\usepackage{caption}
\usepackage{epsfig,here}
\usepackage{subfigure,here}

\newtheorem{theorem}{Theorem}[section]
\newtheorem{lemma}[theorem]{Lemma}
\newtheorem{proposition}[theorem]{Proposition}
\newtheorem{definition}[theorem]{Definition}
\newtheorem{remark}[theorem]{Remark}
\newtheorem{example}[theorem]{Example}

\def\p{\partial}

\def\p{\partial}

\def\<{\langle}
\def\>{\rangle}

\def\ep{\epsilon}

\def\Om{\Omega}

\def\p{\partial}

\newcommand{\mfH}{\mathbf{H}}

\newcommand{\mfM}{\mathbf{M}}
\newcommand{\mfR}{\mathbf{R}}
\newcommand{\mfV}{\mathbf{V}}
\newcommand{\mfRV}{\mathbf{R}\mathbf{V}}

\newcommand{\mcH}{\mathcal{H}}

\newcommand{\mcL}{\mathcal{L}}

\newcommand{\mcV}{\mathcal{V}}

\newcommand{\mcRV}{\mathcal{R}\mathcal{V}}

\newcommand{\ra}{\rightarrow}

\newcommand{\rd}{{\rm d}}

\newcommand{\vvert}{{\vert\vert}}

\providecommand{\norm}[1]{\lVert#1\rVert}

\newcommand{\eq}[1]{\begin{equation}\begin{alignedat}{2} #1 \end{alignedat}\end{equation}}

\numberwithin{equation} {section}

\begin{document}
	
	\title[Boundary Maximum Principle]{A boundary maximum principle for stationary pairs of varifolds with fixed contact angle}
	
	\author[Zhang]{Xuwen Zhang}
	\address[X.Z]{School of Mathematical Sciences\\
		Xiamen University\\
		361005, Xiamen, P.R. China
		\newline\indent Institut f\"ur Mathematik, Goethe-Universit\"at, 
		60325, Frankfurt, Germany}
	\email{zhang@math.uni-frankfurt.de}
	
	\begin{abstract}
	{\color{black}In this note, we establish a boundary maximum principle for a class of stationary pairs of varifolds satisfying a fixed contact angle condition in any compact Riemannian manifold with smooth boundary}.	
		\
		
		
	\end{abstract}
	
	\maketitle
	
	

\section{Introduction}\label{Sec-1}
Minimal surfaces—critical points of the area functional with respect to local deformations—are fundamental objects in Riemannian geometry, and attracted the attention of many mathematicians. In this note, we establish a boundary maximum principle for the generalized minimal hypersurfaces in any Riemannian manifolds, having constant contact angle $\theta_0$ with the boundary.

In all follows, let $(N^\ast,g)$ be a smooth, connected, compact $(n+1)$-dimensional Riemannian manifold with non-empty boundary $\p N^\ast$.
With a slight abuse of notation, we also use $\<\cdot,\cdot\>$ to denote the Riemannian metric of $N^\ast$ and denote by $\nabla$ the Levi-Civita connection of $N^\ast$.
Let $\nu_{N^\ast}$ denote its unit normal along $\p N^\ast$, pointing into $N^\ast$.
For any smooth, compact, properly embedded hypersurface $S\subset N^\ast$
whose boundary lies in $\p N^\ast$, fix an orientation given by the unit normal vector field $\nu_S$, and
let $\Omega$ be the closure of the enclosed region of $S$ with $\p N^\ast$ such that $\nu_S$ points inside $\Omega$, set $T=\p\Omega\cap\p N^\ast$.
See Fig. \ref{Fig-1} for illustration, where $\mu,\bar\mu$ denote the inwards pointing unit conormals of $S\cap T$ in $S$ and $T$ respectively.

Let $A^S$ denote the shape operator of $S$ in $N^\ast$ with respect to $\nu_S$, i.e., $A^S(u)=-\nabla_u\nu_S$ for any $u\in\Gamma(TS)$. We say that $S$ is strongly mean convex at a point $p\in S$, if
\eq{
    \kappa_1+\ldots+\kappa_n>0,
}
where $\kappa_1\leq\ldots\leq\kappa_n$ are the principal curvatures of $A^S$ at $p$.
\begin{figure}[h]
	\centering
	\includegraphics[height=7cm,width=11cm]{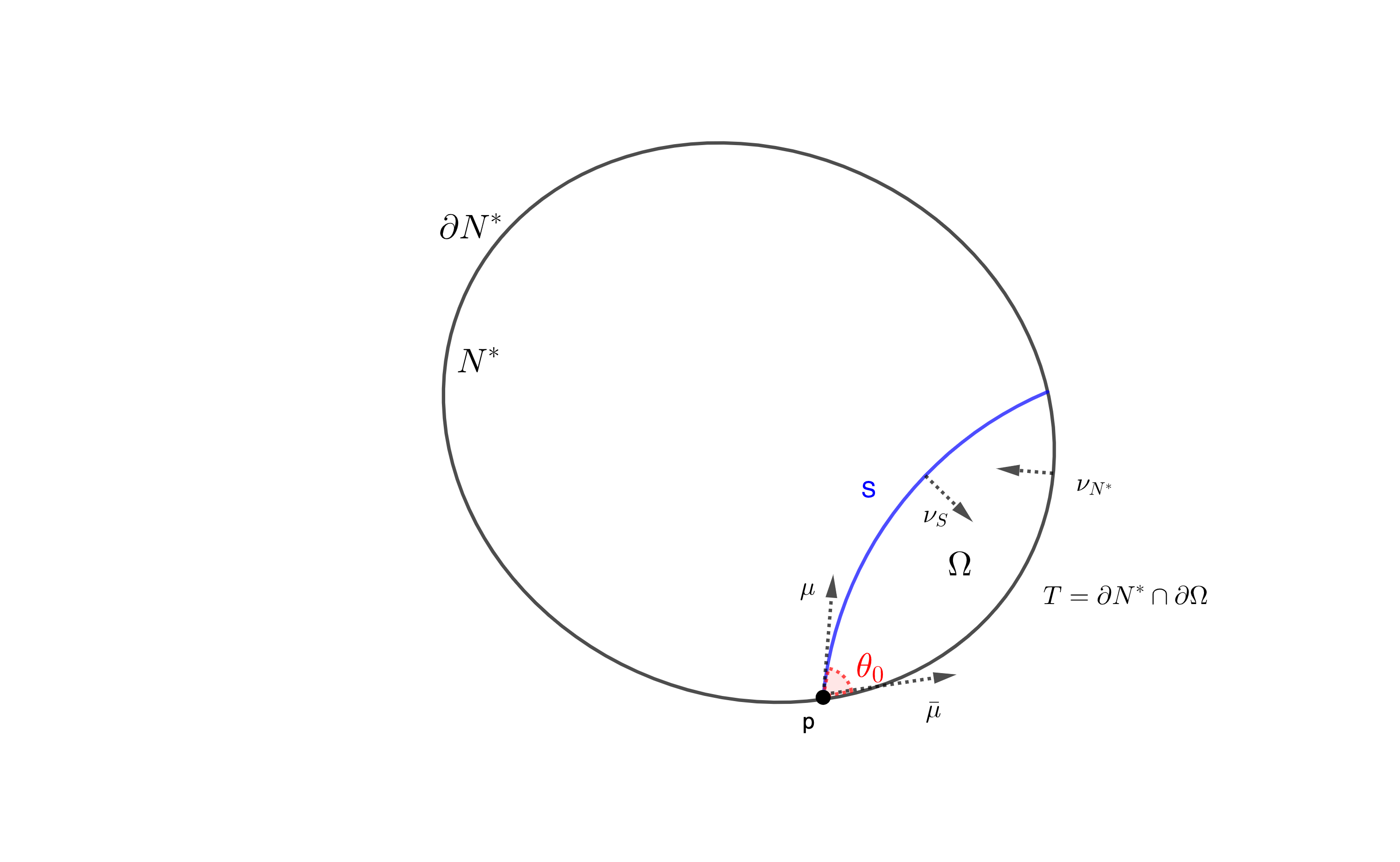}
	\caption{Hypersurface $S$ having constant contact angle $\theta_0$ with the boundary $\p N^\ast$.}
	\label{Fig-1}
\end{figure}

Our main result is the following boundary maximum principle, established in the context of varifolds with fixed contact angle, we refer to Section \ref{Sec-2} for the precise definition and statement.
\subsection{Main Result}
\begin{theorem}\label{Thm-MP-FixedContactAngle}
    Given $\theta_0\in(0,\pi/2)$,
    let S be a smooth, compact, properly embedded hypersurface, meeting $\p N^\ast$ with a constant angle $\theta_0$, that is, $\left<\nu_S,\nu_{N^\ast}\right>=-\cos\theta_0$ along $\p N^\ast\cap S$. Suppose $S$ is strongly mean convex at a point $p\in\p S$.
    
    Then, for any fine $\theta_0$-stationary pair $(V,W)\in\mcV_n(N^\ast)\times\mcV_n(\p N^\ast)$,
    $p$ is not contained in the support of $\vert\vert V\vert\vert$, if one of the following cases happens:\footnote{As pointed out to us by a referee, the following two conditions are sufficient also if they are locally true at $p$.}
    \begin{enumerate}[i.]
        \item $\vvert V\vvert$ is supported in $\Om$, $\vvert W\vvert$ is supported in $T$, and $\vvert V\vvert(\p N^\ast)=0$; 
        \item
        $\vvert V\vvert$ is supported in $\Om$, $\vvert W\vvert$ is supported in $T$, and
        $\p N^\ast$ is mean convex in $N^\ast$.
    \end{enumerate}
\end{theorem}

The maximum principle for minimal submanifolds has been proved in various context. The interior maximum principle for $C^2$-hypersurfaces is a direct consequence of the well-known Hopf's boundary point lemma \cite[Lemma 3.4]{GT01}.
It is then generalized to arbitrary codimension by Jorge-Tomi \cite{JT03}.
In the non-smooth case, White \cite{White10} established the interior maximum principle in the context of minimal varieties, in any codimension.
Recently, Li-Zhou generalized the main result of \cite{White10} to the free boundary setting, they established a boundary maximum principle for free boundary minimal varieties (free boundary stationary varifolds), in arbitrary codimension (\cite[Theorem 1.2]{LZ21}).

As argued in \cite{LZ21}, in the smooth, codimension-1 case, the boundary maximum principle for free boundary hypersurface amounts to be a simple application of Hopf's lemma. 
{\color{black}Meanwhile, one can also derive a boundary maximum principle for a generic contact angle $\theta_0\in(0,\pi)$ by virtue of Hopf’s lemma, see e.g. \cite[Lemma 1.13]{LZZ21}.
Therefore our main result Theorem \ref{Thm-MP-FixedContactAngle} serves as a generalization
of this classical result and of course as an extension of \cite[Theorem 1.2]{LZ21} from $\theta_0=\pi/2$ to $\theta_0\in(0,\pi/2)$.}
\begin{figure}[h]
	\centering
	\includegraphics[height=7cm,width=11cm]{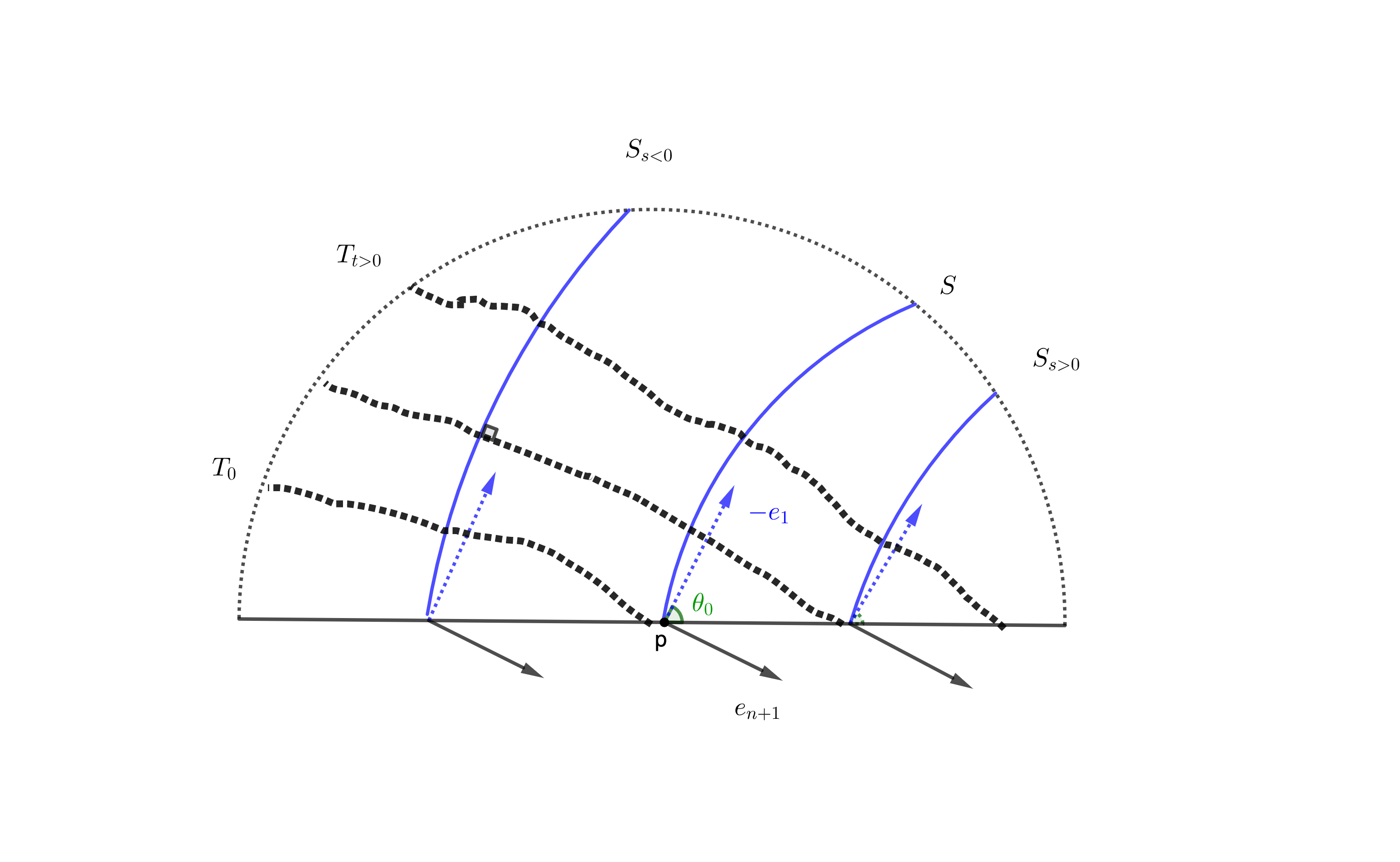}
	\caption{Local orthogonal foliations and orthonormal frame.}
	\label{Fig-2}
\end{figure}

Our strategy of proof follows largely from \cite{LZ21}.
In the free boundary case ($S$ meets $\p N^\ast$ orthogonally), Li-Zhou managed to prove their main result by a contradiction argument. Precisely, they constructed a test vector field $X$, which strictly decreases the first variation of the free boundary stationary varifold.
To construct such $X$, they first constructed local orthogonal foliations near $p$ (\cite[Lemma 2.1]{LZ21}).
By virtue of such foliations, they found a local orthonormal frame ${e_1,\ldots,e_{n+1}}$ of $N^\ast$ near $p$, see \cite[Figure 2]{LZ21} also Fig. \ref{Fig-2} for illustration. The key point is that, locally near the free boundary, for any $q\in\p N^\ast$, there holds $e_{n+1}(q)\in T_q\p N^\ast$, which motivates their choice of test vector field $X$.
{\color{black}In our case, intuitively we would like to seek some tangential variation $X$ to test the first variation for $\theta$-stationary pair of varifolds (as in Definition \ref{Defn-contactangle}).
However, we only managed to prove a weaker boundary maximum principle in a preliminary version of this manuscript \cite[Theorem 1.1]{Zhang22}.
In fact, the strong maximum principle does not hold for a general stationary pair of varifolds (see Examples \ref{Exam-counterEx}, \ref{Exam-A1}).
Motivated by this, we introduce a new class of pairs of varifolds, called \textit{fine $\theta$-stationary pair} (defined in Definition \ref{Defn-finepair}), which allows us to test not only the tangential (with respect to $\p N^\ast$) variation but also the normal one (with respect to $S$), as what Li-Zhou have done in \cite{LZ21}.
Testing the stationarity of the fine pair $(V,W)$ with normal variation,
we find: the strictly mean convexity of $S$ at $p$ forces $X$ to increase the normal variation of $V$, which violates the stationarity of the fine pair $(V,W)$.}
\subsection{Organization of the Paper} In Section \ref{Sec-2}, we briefly recall some definitions from geometric measure theory and give a precise definition of the fine $\theta$-stationary pair of varifolds.
In Section \ref{Sec-3}, we prove our main result Theorem \ref{Thm-MP-FixedContactAngle}.

\subsection{Acknowledgements}

The author would like to thank the anonymous referees for many valuable comments. A special thanks goes to Gaoming Wang, for the counter example (Example \ref{Exam-A1}) shown to the author by him when a preliminary version of this manuscript is under review, which somewhat motivated the definition of finte $\theta$-pair. 


\section{Varifolds with Fixed Contact Angle}\label{Sec-2}

Let us begin by recalling some basic concepts of varifolds, we refer to \cite[Chapter 8]{Allard72,Pitts81,Sim83} for detailed accounts.
\subsection{Varifolds}
The space of \textit{$n$-varifolds} in $\mfR^{L}$, denoted by $\mfV_n(\mfR^{L})$, is the set of all positive Radon measures on the Grassmannian $\mfR^{L}\times G(L,n)$ equipped with the weak topology. The \textit{weight} and \textit{mass} of a varifolds $V\in\mfV_n(\mfR^{L})$ is denoted respectively by $\vvert V\vvert$ and $\mfM(V):=\vvert V\vvert(\mfR^{L})$. For any Borel set $A\subset\mfR^{L}$, we denote by $V\llcorner A$ the \textit{restriction} of $V$ to $A\times G(L,n)$. The \textit{support} of $V$, ${\rm spt}\vvert V\vvert$, is the smallest closed subset $B\subset\mfR^{L}$ such that $V\llcorner(\mfR^{L}\setminus B)=0$.
For any $C^1$ map $f:\mfR^{L}\ra\mfR^{L}$, the continuous \textit{pushforward} map $f_\#:\mfV_n(\mfR^{L})\ra\mfV_n(\mfR^{L})$ is defined as in \cite[2.1(18)(h)]{Pitts81}.
We denote by $\mfRV_n(\mfR^{L})$ the set of \textit{rectifiable $n$-varifolds} in $\mfR^{L}$, see \cite[2.1(18)(d)]{Pitts81}.

Let us proceed and define varifolds in $N^\ast$, by virtue of the Nash embedding theorem, we can assume that $N^\ast$ is isometrically embedded as a closed subset of some $\mfR^L$.
Here we follow the notations in \cite{Pitts81} (which is slightly different from \cite{Allard72}, see also \cite[Section 2.2]{LZ21-a}).
Since $N^\ast$ is a submanifold of class $1$ of $\mfR^L$, we define the Grassmann bundle by $G_n(N^\ast)=\left(N^\ast\times G(L,n)\right)\cap\{(x,P):P\subset T_xN^\ast\}$; $G_n(\p N^\ast)$ is understood in the same way, and since $\p N^\ast$ is $n$-dimensional, we have $G_n(\p N^\ast)=\{(x,T_x\p N^\ast):x\in\p N^\ast\}$.

In this note,
we mainly work with the following spaces of vector fields,
\eq{
&\mathfrak{X}(\mfR^L)\coloneqq\{\text{the space of }C^1\text{-vector fields on }\mfR^L\},\\
&\mathfrak{X}(N^\ast)\coloneqq\{X\in\mathfrak{X}(\mfR^L):X(p)\in T_pN^\ast\text{ for all }p\in N^\ast\},\\
&\mathfrak{X}_t(N^\ast)\coloneqq\{X\in \mathfrak{X}(N^\ast):X(p)\in T_p(\p N^\ast)\text{ for all }p\in \p N^\ast\}.
}
Notice that at any $p\in\p N^\ast$, $T_pN^\ast$ is exactly the $n$-dimensional half-space in $\mfR^L$ with boundary $T_p(\p N^\ast)$.

We define the space of rectifiable $n$-varifolds in $N^\ast$, denoted by $\mcRV_n(N^\ast)$, to be the set of all rectifiable $n$-varifiolds in $\mfR^L$ with ${\rm spt}\vvert V\vvert\subset N^\ast$.
Moreover, $\mcV_n(N^\ast)$ is defined to be the closure, in the weak topology, of $\mcRV_n(N^\ast)$.
Note that in general, $\mcV_n(N^\ast)$ is a proper subset of $\mfV_n(\mfR^L)\cap\{V:{\rm spt}\vvert V\vvert\subset N^\ast\}$, and in fact, if $V\in\mcV_n(N^\ast)$, then one has (see \cite[2.1(18)(g)]{Pitts81})
\eq{
    V(G_n(\mfR^L)\setminus G_n(N^\ast))=0.
}
Similar property holds for those $W\in\mcV_n(\p N^\ast)$.

Let $V\in\mcV_n(N^\ast)$, if $X\in\mathfrak{X}(N)$ generates a one-parameter family of diffeomorphisms $\phi_t$ of $\mfR^L$ with $\phi_t(N^\ast)\subset N^\ast$ (at a point p on $\partial N^\ast$, one considers the tangent space $T_p N^\ast$ as the half $(n+1)$-space obtained by the blow-up of $N^\ast$ at $p$), then $(\phi_t)_\#V\in\mcV_n(N^\ast)$ and one can consider its first variation along $X$ \cite[(4.2),(4.4)]{Allard72}
\eq{\label{defn-1st-variationformula}
    \delta V[X]\coloneqq\frac{\rd}{\rd t}\mid_{t=0}\mfM((\phi_t)_\# V)
    =\int_{G_n(N^\ast)}{\rm div}_PX(x)\rd V(x,P),
}
here ${\rm div}_PX(x)=D_{e_i}X\cdot e_i$, where $\{e_1,\ldots,e_n\}\subset P$ is any orthonormal basis.
In particular, for $W\in\mcV_n(\p N^\ast)$, we have
\eq{\label{eq-1stvariation-W}
    \delta W[X]
    =\int_{G_n(\p N^\ast)}{\rm div}_PX(x)\rd W(x,P)
    =\int_{\p N^\ast}{\rm div}_{\p N^\ast}X(x)\rd\vvert W\vvert(x).
}
\subsection{Contact Angle Condition for Varifolds}
Let us first introduce the contact angle condition for varifolds, which, to the author's knowledge, was brought up in \cite{KT17} formally, and then extended to a weaker form in \cite{MP21}.

\begin{definition}[Contact angle condition, {\cite[Definition 3.1]{MP21}}]\label{Defn-contactangle}
\normalfont
Given $\theta\in(0,\pi)$, we say that the pair $(V,W)\in\mcV_n(N^\ast)\times\mcV_n(\p N^\ast)$ satisfies the contact angle condition $\theta$, if there exists a $\vert\vert V\vert\vert$-measurable vector field $\mfH\in\mcL^1(N^\ast,\vert\vert V\vert\vert)$ with $\mfH(x)\in T_x\p N^\ast$ for $\norm{V}$-a.e. $x\in\p N^\ast$, such that for every $X\in\mathfrak{X}_t(N^\ast)$, it holds \footnote{We note that our definitions of varifolds are different from that of \cite{DeMasi21,MP21}, where a $k$-varifold on $A\subset\mfR^L$ is defined to be a positive Radon measure on $A\times G(L,k)$.
Here we can rewrite the second term by virute of \eqref{eq-1stvariation-W}.
}
\eq{\label{def-contactangle}
    \delta_{F_{\theta}}(V,W)[X]
    \coloneqq&\int_{G_n(N^\ast)}{\rm div}_PX(x)\rd V(x,P)-\cos\theta\int_{G_n(\p N^\ast)}{\rm div}_PX(x)\rd W(x,P)\\
    =&-\int_{N^\ast}\left<X(x),\mfH(x)\right>\rd\vert\vert V\vert\vert(x).
}
In particular, we say that $(V,W)$ is a $\theta$-stationary pair if in addition, $\mfH=0$ for a.e. $x\in{\rm spt}\vvert V\vvert$.
\end{definition}
An important proposition \footnote{As noted before, our definitions of varifolds are different from \cite{MP21}.
However, the proof therein, which is based on \cite[Corollary 4.6]{DeMasi21} and \cite[Proposition 3.17]{DeMasi22}, works finely for our definitions.} for the pair of varifolds with fixed contact angle is that they have bounded first variation and satisfies the following first variation formula.
\begin{proposition}[{\cite[Proposition 3.1]{MP21}}]\label{Prop-MP21-3-1}
Given $\theta\in[\pi/2,\pi)$, let $(V,W)\in\mcV_n(N^\ast)\times \mcV_n(\p N^\ast)$ have fixed contact angle $\theta$.
Then $V-\cos\theta W$ has bounded first variation.
More precisely, there exists a positive Radon measure $\sigma_V$ on $\p N^\ast$ and a continuous vector field $\tilde \mfH$, such that
\eq{\label{defn-finepair-1}
        \int_{G_n(N^\ast)}&{\rm div}_PX(x)\rd V(x,P)-\cos\theta\int_{G_n(\p N^\ast)}{\rm div}_PX(x)\rd W(x,P)
        =-\int_{N^\ast}\left<X,\mfH\right>\rd\vvert V\vvert\\
        -&\int_{\p N^\ast}\left<X,\tilde \mfH\right>\rd\left(\vvert V\vvert-\cos\theta\vvert W\vvert\right)+\int_{\p N^\ast}\left<X,-\nu_{N^\ast}\right>\rd\sigma_V,\quad\forall X\in\mathfrak{X}(N^\ast),
}
where $\mfH$ is as in Definition \ref{Defn-contactangle}; $\tilde \mfH$ is the mean curvature of $\p N^\ast$ in $N^\ast$, given by 
\eq{
\tilde \mfH(x)=\nu_{N^\ast}(x){\rm div}_{\p N^\ast}(-\nu_{N^\ast}(x)),\text{ for every }x\in\p N^\ast,
}
and we denote by $\tilde H(x)={\rm div}_{\p N^\ast}(-\nu_{N^\ast}(x))$ so that $\tilde \mfH(x)=\tilde H(x)\nu_{N^\ast}(x)$.
\end{proposition}

If $V$ is a $C^2$-hypersurface in $N^\ast$ (similar with the hypersurface $S$ considered in the introduction), we denote by $\Omega$ the enclosed region of $V$ and $\p N^\ast$, $T:=\p\Omega\cap\p N^\ast$ is the wetting hypersurface with respect to $V$ and $\bar\mu$ is the inwards pointing unit conormal of $V\cap T$ in $T$,
then the first variation of $V$ with respect to the direction $X\in\mathfrak{X}(N^\ast)$ is just
\eq{\label{eq-hypersurface1}
    \delta V(X)
    \coloneqq\frac{\rd}{\rd t
    }\mid_{t=0}{\rm Area}(\psi_t(V))=\int_V{\rm div}_VX\rd\mcH^n\\
    =-\int_V\left<X,H\right>\rd\mcH^n-\int_{\p V}\left<X,\mu\right>\rd\mcH^{n-1},
}
where $\psi_t$ is the flow of $X$ at the time $t$, $H$ is the inwards pointing mean curvature vector of $V$ and $\mu$ is the inwards pointing unit conormal of $\p V$ in $V$.
If $V$ meets $\p N^\ast$ with constant contact angle $\theta$, notice that along $\p N^\ast$, we have $\mu=\cos\theta\bar\mu+\sin\theta\nu_{N^\ast}$ (see Fig. \ref{Fig-1}), we thus obtain
\eq{\label{eq-hypersurface2}
    \int_{\p V}\left<X,\mu\right>\rd\mcH^{n-1}
    =&\cos\theta\int_{\p V}\left<X,\bar\mu\right>\rd\mcH^{n-1}+\sin\theta\int_{\p V}\left<X,\nu_{N^\ast}\right>\rd\mcH^{n-1}\\
    =&{\color{black}-\cos\theta\left(\int_T{\rm div}_{\p N^\ast}X\rd\mcH^n+\int_T\left<X,\tilde H\right>\rd\mcH^n\right)}+\sin\theta\int_{\p V}\left<X,\nu_{N^\ast}\right>\rd\mcH^{n-1},
}
where we used the contact angle condition to derive the first equality and the divergence theorem for the second equality.
In this case,
it is clear that the Radon measure $\sigma_V$ in Proposition \ref{Prop-MP21-3-1} is just $\sin\theta\mcH^{n-1}\llcorner(\p V)$.
On the other hand, for $X\in\mathfrak{X}_t(N^\ast)$, one has
\eq{
    \int_T{\rm div}_{\p N^\ast}X\rd \mcH^n
    =\int_{\p V}\left<X,-\bar\mu\right>\rd\mcH^{n-1}
    =\frac{1}{\sin\theta}\int_{\p N^\ast}\left<X,-\bar\mu\right>\rd \sigma_V.
}
Moreover, consider the set $\p N^\ast\setminus T$, it is easy to see that for $X\in\mathfrak{X}_t(N^\ast)$, there holds
\eq{
    \int_{\p N^\ast\setminus T}{\rm div}_{\p N^\ast}X\rd\mcH^n
    =-\int_{\p N^\ast\setminus T}{\rm div}_{\p N^\ast}X\rd\mcH^n.
}

Enlightened by this simple observation, we introduce the following definition that is stronger than Definition \ref{Defn-contactangle}.

\begin{definition}\label{Defn-finepair}
\normalfont
For $\theta\in[\pi/2,\pi)$, let $(V,W)\in\mcV_n(N^\ast)\times \mcV_n(\p N^\ast)$ have fixed contact angle $\theta$ and let $\sigma_V,\mfH,\tilde \mfH$ be as in Proposition \ref{Prop-MP21-3-1}.
We say that $(V,W)$ is a fine $\theta$-pair if there exists $\tilde\mu\in\mcL^1(\p N^\ast,\sigma_V)$ with
$\tilde\mu(x)\in T_x\p N^\ast$ and
$\vert\tilde\mu(x)\vert=1$ for a.e. $x\in{\rm spt}\sigma_V$, such that: for every $X\in\mathfrak{X}_t(N^\ast)$, there holds
\eq{\label{formu-1stvariation-W}
    \int_{G_n(\p N^\ast)}{\rm div}_PX(x)\rd W(x,P)
    =\int_{\p N^\ast}{\rm div}_{\p N^\ast}X\rd\vvert W\vvert
    =\frac{1}{\sin\theta}\int_{\p N^\ast}\left<X,\tilde\mu\right>\rd\sigma_V.
}

For $\theta\in(0,\pi/2)$,
we say that $(V,W)\in\mcV_n(N^\ast)\times \mcV_n(\p N^\ast)$ is a fine $\theta$-pair if there exists $\tilde W\in\mcV_n(\p N^\ast)$ such that
\begin{enumerate}
    \item $(V,\tilde W)$ is a fine $(\pi-\theta)$-pair (in this case $\pi-\theta\in(\pi/2,\pi)$, let $\sigma_V,\mfH,\tilde\mfH,\tilde\mu$ be the resulting notations);
    \item For any $X\in\mathfrak{X}_t(N^\ast)$, there holds
\eq{\label{formu-1stvariation-W^c}
        \int_{\p N^\ast}{\rm div}_{\p N^\ast}X\rd\vvert W\vvert
        =-\int_{\p N^\ast}{\rm div}_{\p N^\ast}X\rd\vvert \tilde W\vvert.
}
\end{enumerate}

In particular, we say that $(V,W)$ is a fine $\theta$-stationary pair if in addition, $\mfH=0$ for a.e. $x\in{\rm spt}\vvert V\vvert$.
\end{definition}

Now we consider the case $\theta\in(0,\pi/2)$ and $(V,W)$ is a fine $\theta$-pair, by definition there exists $\tilde W\in\mcV_n(\p N^\ast)$ such that $(V,\tilde W)$ is a fine $(\pi-\theta)$-pair, and from \eqref{formu-1stvariation-W^c}, \eqref{formu-1stvariation-W}, we see that for any $X\in\mathfrak{X}_t(N^\ast)$,
\eq{\label{formu-1stvariation-W'}
        \int_{\p N^\ast}{\rm div}_{\p N^\ast}X\rd\vvert W\vvert
        =-\int_{\p N^\ast}{\rm div}_{\p N^\ast}X\rd\vvert \tilde W\vvert
        =&-\frac{1}{\sin(\pi-\theta)}\int_{\p N^\ast}\left<X,\tilde\mu\right>\rd\sigma_V\\
        =&\frac{1}{\sin\theta}\int_{\p N^\ast}\left<X,-\tilde\mu\right>\rd\sigma_V.
}
Since $\p N^\ast$ is a smooth hypersurface in $N^\ast$, a standard computation then gives that:
for any $X\in\mathfrak{X}(N^\ast)$,
\eq{\label{defn-finepair-2}
     \int_{\p N^\ast}{\rm div}_{\p N^\ast}X\rd\vvert \tilde W\vvert
     =&\int_{\p N^\ast}{\rm div}_{\p N^\ast}(X^T+X^\perp)\rd\vvert\tilde W\vvert\\
     =&\frac{1}{\sin\theta}\int_{\p N^\ast}\left<X,\tilde\mu\right>\rd\sigma_V-\int_{\p N^\ast}\left<X,\tilde \mfH\right>\rd\vvert \tilde W\vvert,
}
here the tangential and normal part of $X$ are stated with respect to $\p N^\ast$, to derive the second equality we have used \eqref{formu-1stvariation-W} for $\tilde W$ (notice that $\sin(\pi-\theta)=\sin\theta$) 
and the fact that $\tilde\mu(x)\in T_x\p N^\ast$ for a.e. $x\in{\rm spt}\sigma_V$.

If in addition, the fine $\theta$-pair $(V,W)$ is stationary, then \eqref{defn-finepair-1} yields
\eq{
    &\int_{G_n(N^\ast)}{\rm div}_PX(x)\rd V(x,P)-\cos(\pi-\theta)\int_{\p N^\ast}{\rm div}_{\p N^\ast}X\rd\vvert\tilde W\vvert\\
    =&-\int_{\p N^\ast}\left<X,\tilde \mfH\right>\rd\left(\vvert V\vvert-\cos(\pi-\theta)\vvert\tilde W\vvert\right)+\int_{\p N^\ast}\left<X,-\nu_{N^\ast}\right>\rd\sigma_V,
}
taking \eqref{defn-finepair-2} into account, this reads
\eq{\label{defn-finepair-3}
    &\int_{G_n(N^\ast)}{\rm div}_PX(x)\rd V(x,P)\\
    =&\frac{\cos\theta}{\sin\theta}\int_{\p N^\ast}\left<X,-\tilde\mu\right>\rd\sigma_V-\int_{\p N^\ast}\tilde H\left<X,\nu_{N^\ast}\right>\rd\vvert V\vvert+\int_{\p N^\ast}\left<X,-\nu_{N^\ast}\right>\rd\sigma_V.
}

\begin{remark}
\normalfont
In Definition \ref{Defn-finepair}, for the case $\theta\in(0,\pi/2)$ we require the existence of $\tilde W\in\mcV_n(\p N^\ast)$ with desired properties.
Note that these properties clearly hold when $W$ is the naturally induced rectifiable varifold of a smooth compact domain $T\subset\p N^\ast$, since we can simply take $\tilde W$ to be the naturally induced rectifiable varifold of $\p N^\ast\setminus T$.
The statement here remains true if $T$ is a Caccioppoli set (set of finite perimeter) in $\p N^\ast$.
\end{remark}
\begin{remark}\label{Rem-1st-variation}
\normalfont
In Theorem \ref{Thm-MP-FixedContactAngle}, we consider the cases when $V$ is supported in $\Omega$ and $W$ is supported in $T=\p\Omega\cap\p N^\ast$.
Thanks to \eqref{formu-1stvariation-W'}, we know that as a positive Radon measure on $\p N^\ast$, $\sigma_V$ is supported on $T$.
Therefore when testing the first variation \eqref{defn-finepair-3}, it suffice to consider the behavior of $X$ on $\Omega$.
\end{remark}

We end this section by giving an example of a pair of varifolds that satisfies the contact angle condition but violates the conclusion of Theorem \ref{Thm-MP-FixedContactAngle}, revealing the necessity of defining \textit{fine $\theta$-pair}.
The construction in the following example is inspired by another example (Example \ref{Exam-A1}) shown to the author by Gaoming Wang when a preliminary version of this manuscript is reviewed and somewhat becomes the major motivation of this paper and the strong maximum principle for pairs of stationary rectifiable cones derived in \cite[Lemma 2.16]{XZ23}.
\begin{example}\label{Exam-counterEx}
Let $N^\ast$ be the unit ball in the plane $\mfR^2$ centered at the origin, $\p N^\ast$ is then the unit sphere.
Let $p_1,p_2$ be two points on $\p N^\ast$ such that the line segment joining them, say $L$, has contact angle $\frac\pi3$ with $\p N^\ast$.
Let $S$ be a mean convex curve joining $p_1$ and $p_2$ in $N^\ast$ to enclose a domain $\Om$ so that $S$ has a contact angle $\theta_0=\arccos\frac14$ with $\p N^\ast$.
Let $V$ be the naturally induced multiplicity 1 varifold by $L$, let $W$ be the multiplicity 2 varifold induced by $\p\Om\cap\p N^\ast$.
See Fig. \ref{Fig-example}.

Then $(V,W)$ is a $\theta_0$-stationary pair but not a fine $\theta_0$-stationary pair, with $\vvert V\vvert$ supported in $\Om$.
\end{example}

\begin{figure}[H]
	\centering
	\includegraphics[height=8cm,width=15cm]{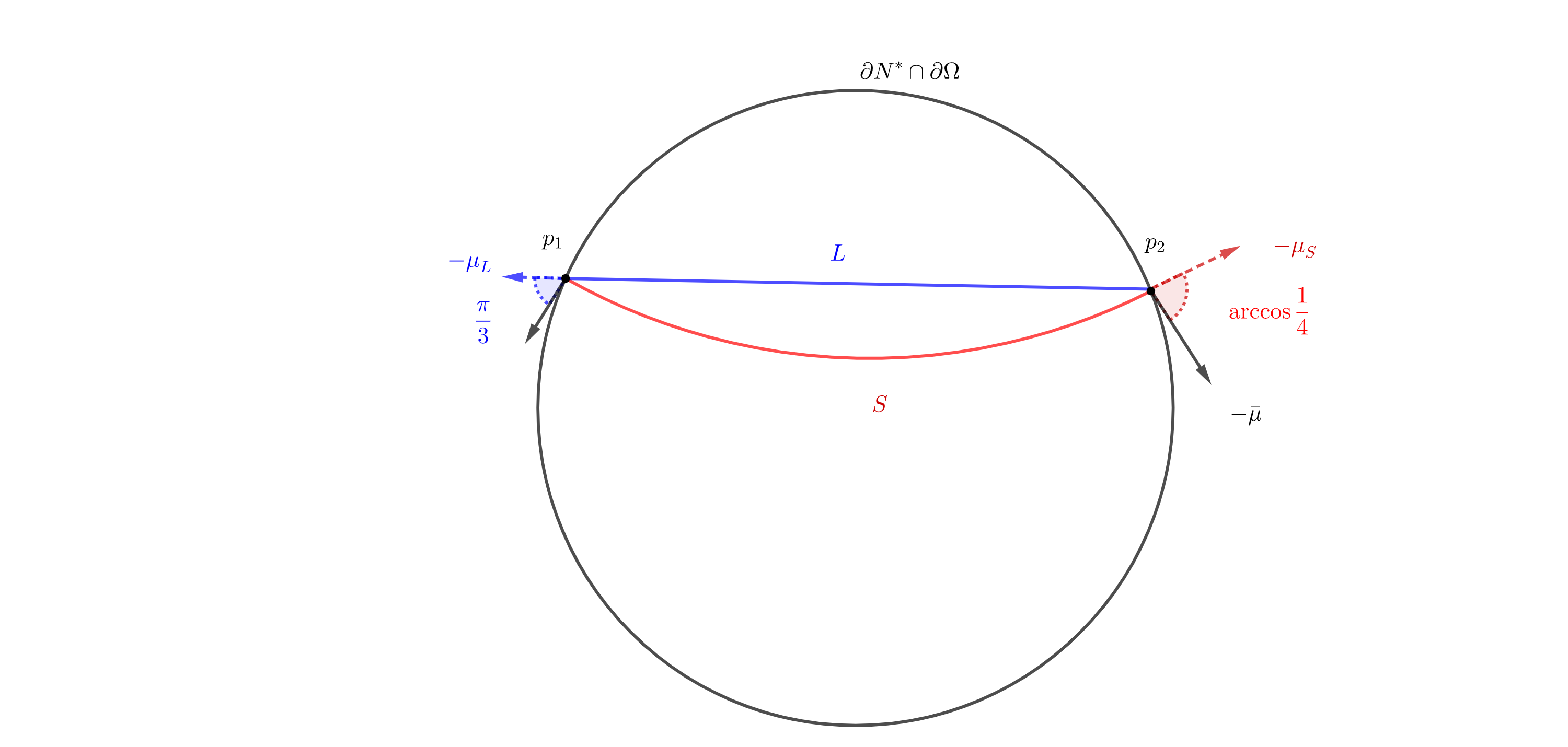}
	\caption{Example 2.6.}
	\label{Fig-example}
\end{figure}

\begin{proof}
A direct computation shows that: for any $X\in\mathfrak{X}(N^\ast)$,
\eq{\label{eq-Exam-1}
    \int_{G_1(N^\ast)}{\rm div}_PX(x)\rd V(x,P)
    =&\int_L{\rm div}_LX(x)\rd\mcH^1(x)
    =\left<X(p_1),-\mu_L(p_1)\right>+\left<X(p_2),-\mu_L(p_2)\right>\\
    =&\frac12\sum_{i=1}^2\left<X(p_i),-\bar\mu(p_i)\right>+\frac{\sqrt{3}}{2}\sum_{i=1}^2\left<X(p_i),-\nu_{N^\ast}(p_i)\right>,
}
where we have used the fact that $-\mu_L(p_i)=\cos\frac{\pi}{3}(-\bar\mu(p_i))+\sin\frac{\pi}{3}(-\nu_{N^\ast}(p_i))$.

Similarly, one has
\eq{\label{eq-Exam-2}
    \int_{G_1(\p N^\ast)}{\rm div}_PX(x)\rd W(x,P)
    =&\int_{\p N^\ast\cap\p\Om}2{\rm div}_{\p N^\ast}X(x)\rd\mcH^1(x)\\
    =&2\sum_{i=1}^2\left<X(p_i),-\bar\mu(p_i)\right>-\int_{\p N^\ast\cap\p\Om}\left<X(x),\tilde \mfH(x)\right>\rd\vvert W\vvert(x).
}
It is then easy to see that for any $X\in\mathfrak{X}_t(N^\ast)$ (recall that $\cos\theta_0=\frac14$), 
\eq{
    \int_{G_1(N^\ast)}{\rm div}_PX(x)\rd V(x,P)-\cos\theta_0\int_{G_1(\p N^\ast)}{\rm div}_PX(x)\rd W(x,P)=0,
}
which shows that $(V,W)$ is a $\theta_0$-stationary pair.

To see that $(V,W)$ is not a fine $\theta_0$-stationary pair, let us consider the multiplicity 2 varifold induced by $\p N^\ast\setminus\p\Om$, denoted by  $\tilde W$.
Clearly for any $X\in\mathfrak{X}_t(N^\ast)$, we have
\eq{\label{eq-Exam-3}
    \int_{\p N^\ast}{\rm div}_{\p N^\ast}X\rd\vvert W\vvert
    =-\int_{\p N^\ast}{\rm div}_{\p N^\ast}X\rd\vvert \tilde W\vvert,
}
and $(V,\tilde W)$ is a $(\pi-\theta_0)$-stationary pair.
It then follows from \eqref{eq-Exam-1}, \eqref{eq-Exam-2}, \eqref{eq-Exam-3} and Proposition \ref{Prop-MP21-3-1} that $\sigma_V$ in this case is given by $\frac{\sqrt{3}}{2}\mcH^0\llcorner\{p_1,p_2\}$.
However, notice that for any $X\in\mathfrak{X}_t(N^\ast)$,
\eq{
    \int_{G_1(\p N^\ast)}{\rm div}_PX(x)\rd \tilde W(x,P)
    =&\int_{\p N^\ast\setminus\p\Om}2{\rm div}_{\p N^\ast}X(x)\rd\mcH^1(x)\\
    =&2\left(\left<X(p_1),\bar\mu(p_1)\right>+\left<X(p_2),\bar\mu(p_2)\right>\right)\\
    \neq&\frac{1}{\frac{\sqrt{15}}{4}}\int_{\p N^\ast}\left<X,\bar\mu\right>\rd\sigma_V,
}
and hence $(V,W)$ is not a fine $\theta_0$-stationary pair. Clearly $\vvert V\vvert$ is supported in $\Om$, which completes the proof.
\end{proof}


\section{Proof of Theorem \ref{Thm-MP-FixedContactAngle}}\label{Sec-3}
As illustrated in Section \ref{Sec-1}, we need the following foliations, see \cite[Lemma 2.1]{LZ21} for the free boundary case.
To prove the following lemma, we will exploit the Fermi coordinate system at $p$.
For discussions on Fermi coordinate, see for example \cite[Section 6]{GLZ20} and \cite[Appendix A]{LZ21-a}.
\begin{lemma}\label{Lem-foliation}
For any properly embedded hypersurface $S$, having constant contact angle $\theta_0\in(0,\pi/2)$ with $\p N^\ast$, there exists a costant $\delta>0$; a neighborhood $U\subset N^\ast$ containing $p\in S\cap\p N^\ast$; and foliations $\left\{S_s\right\}$, $\left\{T_t\right\}$, with $s\in(-\delta,\delta), t\in(0,\delta)$, of $U$, $U\cap\Om$, respectively; such that $S_0=S\cap U$, and $S_s$ intersects $T_t$ orthogonally for every $s$ and $t$. In addition, each hypersurface $S_s$ meets $\p N^\ast$ with constant contact angle $\theta_0$. 
\end{lemma}
\begin{proof}
We first extend $S$ locally near $p$ to a foliation $\left\{S_s\right\}$ such that each $S_s$ meets $\p N^\ast$ with constant contact angle $\theta_0$. This can be done by a simple modification of \cite[Lemma 2.1]{LZ21}.

Let $(x_1,\ldots,x_{n+1})$ be a local Fermi coordinate system of $N^\ast$ centered at $p$, such that $x_1={\rm dist}_{N^\ast}(\cdot,\p N^\ast)$. Furthermore, we assume that $(x_2,\ldots,x_{n+1})$ is a local Fermi coordinate system of $\p N^\ast$, relative to the hypersurface $S\cap\p N^\ast$; that is, $x_{n+1}$ is the signed distance in $\p N^\ast$ from $S\cap\p N^\ast$.

In the rest of this paper, we denote by $B_{r_0}^+=\{x_1^2+\ldots x_n^2<r_0^2\mid x_1\geq0,x_{n+1}=0\}$ the $n$-dimensional half ball in the Fermi coordinate.
Since $S$ meets $\p N^\ast$ with a constant contact angle $\theta_0\in(0,\pi)$, we can express $S$ in such local coordinates as the graph $x_{n+1}=f(x_1,\ldots,x_n)$ of a function $f$ defined on a half ball $B_{r_0}^+$, such that $f=0$ along $B^+_{r_0}\cap\left\{x_1=0\right\}$. Moreover, due to the contact angle condition, we can carry out the following computation, see also \cite[Section 7.1]{CM11} for a detailed computation of minimal graphs on manifolds.

First we fix some notations. Let $g_{ij}$ denote the metric on $N^\ast$ in the local Fermi coordinate $(x_1,\ldots,x_{n+1})$. Set $\bar e_i$ to be the vector field $\frac{\p}{\p x_i}$ so that $\left<\bar e_i,\bar e_j\right>=g_{ij}$.
For simplicity, we define a positive smooth function {\color{black}$W_f$} by
\eq{\label{defn-W_f}
    W_f^2(x_1,\ldots,x_{n})=g^{n+1,n+1}+\sum_{i,j=1}^ng^{ij}(x_1,\ldots,x_n,f(x_1\,\ldots,x_n))\frac{\p f}{\p x_i}\frac{\p f}{\p x_j}-2\sum_{l=1}^ng^{l,n+1}\frac{\p f}{\p x_l}.
}
Now, let $\nu$ denote the outwards pointing unit normal of $S$,  
computing as \cite[(7.11)]{CM11}, we obtain
\eq{\label{eq-Lem1-1}
    \left<\nu,\bar e_i\right>=\frac{1}{{\color{black}W_f}}\frac{\p f}{\p x_i},\quad i=1,\ldots,n.
}
In particular, since $S$ meets $\p N^\ast$ with contact angle $\theta_0$, we have $\left<\nu,\bar e_1\right>=\cos\theta_0$ along $\left\{x_1=0\right\}$, and hence \eqref{eq-Lem1-1} yields
\eq{\label{eq-Lem1-2}
    \frac{\p f}{\p x_1}(0,x_2,\ldots,x_n)=\cos\theta_0 W_f(0,x_2,\ldots,x_n)\quad\text{on }\left\{x_1=0\right\}.
}
Note that on $N^\ast$ we have $g_{11}=1, g_{1k}=0$ for $k=2,\ldots,n+1$ since $x_1={\rm dist}_{N^\ast}(\cdot,\p N^\ast)$, and on $\{x_1=0\}\subset\p N^\ast$ we have $g_{n+1,n+1}=1,g_{n+1,l}=0$ for $l=2,\ldots,n$ since $x_{n+1}$ is the signed distance function in $\p N^\ast$ from $S\cap\p N^\ast$.
Recall also that $0=f(0,x_2,\ldots,x_n)$,
\eqref{defn-W_f} when restricted to $\{x_1=0\}$ thus reads
\eq{\label{eq-W_f}
    W^2_f(0,x_2\ldots,x_n)
    =1+\left(\frac{\p f}{\p x_1}(0,x_2,\ldots,x_n)\right)^2,
}
since the other partial derivatives of $f$ vanish.
Combining this with \eqref{eq-Lem1-2}, we get
\eq{
    \frac{\p f}{\p x_1}(0,x_2,\ldots,x_n)=\cot\theta_0,\quad W_f(0,x_2,\ldots,x_n)=\frac{1}{\sin\theta_0}.
}
In view of this, the translated graphs
$$x_{n+1}=f(x_1,\ldots,x_n)+s\eqqcolon f_s(x_1,\ldots,x_n)$$
then gives a local foliation $\left\{S_s\right\}$ near $p$,
and we can show that each leaf $S_s$ is a hypersurface in $N^\ast$ which meets $\p N^\ast$ with constant contact angle $\theta_0$ along its boundary $S_s\cap\p N^\ast$.
Indeed,
a direct computation gives that
\eq{
    \frac{\p f_s}{\p x_1}(0,x_2,\ldots,x_n)
    =&\frac{\p f}{\p x_1}(0,x_2,\ldots,x_n)
    =\cot\theta_0.
}
Since $f_s=s$ along $\{x_1=0\}$, we may argue as
\eqref{eq-W_f} to find that along $\{x_1=0\}$,
\eq{
    W_{f_s}^2(0,x_2,\ldots,x_n)
    =&1+\left(\frac{\p f_s}{\p x_1}(0,x_2,\ldots,x_n)\right)^2
    =\frac{1}{\sin^2\theta_0}.
}
It is then easy to see that along $\{x_1=0\}$,
\eq{
    \frac{\p f_s}{\p x_1}(0,x_2,\ldots,x_n)=\cos\theta_0 W_{f_s}(0,x_2,\ldots,x_n),
}
which implies that $S_s$ touches $\p N^\ast$ with constant contact angle $\theta_0$ according to \eqref{eq-Lem1-2}.

Next, the construction of $\left\{T_t\right\}$ which is orthogonal to every leaf of $\left\{S_s\right\}$ follows from \cite[Lemma 2.1]{LZ21}, we include the details here for readers' convenience.
Let $q\in N^\ast$ be a point near $p$ which lies on the leaf $S_s$.
We define $\nu(q)$ to be a unit vector normal to the hypersurface $S_s$.
By a continuous choice of $\nu$ it gives a smooth unit vector field in a neighborhood of $p$.
Since $\nu$ is nowhere vanishing near $p$, the integral curves of $\nu$ gives a local $1$-dimensional foliation of $N^\ast$ naer $p$.
The desired foliation $\{T_t\}$ is obtained by putting together these integral curves.
Precisely,
let $\Gamma_t\subset S$ be the parallel hypersurface in $S$ which is of distance $t>0$ away from $S\cap \p N^\ast$.
For $t\geq0$, set $T_{t}$ to be the union of all the integral curves of $\nu$ which pass through $\Gamma_{t}$.
Then, by zooming in at $p$, we obtain a small $\delta>0$, and a small set $U\cap\Om$
(see Fig. \ref{Fig-2} for illustration), which is indeed foliated by $\{T_t\}_{t\in(0,\delta)}$.
On the other hand, $\{S_s\}_{s\in(-\delta,\delta)}$ apparently foliates $U$.
This completes the proof. 
\end{proof}

The local orthogonal foliation in Lemma \ref{Lem-foliation} yields the following orthonormal frame of $\Omega$ near $p$, which is needed in our proof of Theorem \ref{Thm-MP-FixedContactAngle}.
\begin{lemma}[{\cite[Lemma 2.2]{LZ21}}]\label{Lem-ONF}
Let $\{e_1,\ldots,e_{n+1}\}$ be a local orthonormal frame of $\Omega$ near $p$, such that at each $q\in S_s\cap T_t$, $e_1(q)$ and $e_{n+1}(q)$ is normal to $S_s\cap T_t$ inside $S_s$ and $T_t$, respectively. In particular, we choose $e_{n+1}$ so that $e_{n+1}=\nu_S$ on $S_0$; $-e_1$ points into $N^\ast$ along $\p N^\ast$.
\end{lemma}

\begin{proof}[Proof of Theorem \ref{Thm-MP-FixedContactAngle}]
As mentioned in the introduction, we want to construct a test vector field $X$, having its support arbitrarily close to $p$ in $\Om$.
This is done in the following manner.

\noindent{\bf Step 1. Constructing a hypersurface $S'$ in $N^\ast$, which touches $\Omega$ from outside up to second order at $p$.}

For every $\ep>0$ small, we define
\eq{
    \Gamma=\left\{x\in\p N^\ast:{\rm dist}_{\p N^\ast}(x,\p S)=\ep{\rm dist}_{\p N^\ast}(x,p)^4\right\},
}
which is an $(n-1)$-dimensional hypersurface in $\p N^\ast$ and is smooth in a neighborhood of $p$.
It has been proved in \cite[Section 3, Claim 1]{LZ21} that $\Gamma$ indeed touches $\p S$ from outside $T$ up to second order at $p$.

Now we extend $\Gamma$ to our desired hypersurface $S'$ in $N^\ast$. The construction is as follows. Let $(x_1,\ldots,x_{n+1})$ be a Fermi coordinate system centered at $p$ as in Lemma \ref{Lem-foliation} so that
\begin{enumerate}
    \item $\left\{x_1\geq0\right\}\subset N^\ast$,
    \item $\left\{x_{n+1}=f(x_1,\ldots,x_n)\geq0\right\}\subset S$,
    \item $\left\{x_{n+1}\geq f(x_1,\ldots,x_n)\right\}\subset \Omega$,
    \item 
    $\left\{x_1=x_{n+1}=0\right\}\subset\Gamma$.
\end{enumerate}
Then, we do a slight modification of the Fermi coordinate $(x_1,\ldots,x_{n+1})$ by further requiring $x_{n+1}$ to agree with the signed distance function from $\Gamma$ in $\p N^\ast$, and denote this coordinate by $(x_1,\tilde x_2,\ldots,\tilde x_{n+1})$, correspondingly, $S$ is expressed as the local graph $\tilde x_{n+1}=\tilde f(x_1,\tilde x_2,\ldots,\tilde x_n)$.
The fact that $\Gamma$ touches $\p S$ from outside $T$ at $p$ implies: $\tilde f(0,\tilde x_2,\ldots,\tilde x_n)\geq0$, with equality holds only at the origin.

In this new Fermi coordinate, we can proceed our construction of $S'$.
Let ${\vec0}$ denote the origin of the Fermi coordinate chart centered at $p$,
we denote by $\tilde g$ the metric in this new coordinate, $\tilde W$ the counterpart of $W$ (given by \eqref{defn-W_f}) in this coordinate.
We set $S'$ to be the graph $\tilde x_{n+1}=u(x_1,\tilde x_2,\ldots,\tilde x_n)$ of the smooth function $u$, defined by
\eq{
    u(x_1,\tilde x_2,\ldots,\tilde x_n)=x_1\cot\theta_0
    +\frac{x_1^2}{2}\frac{\p^2 f}{\p x_1^2}(\vec 0)+\frac{x_1^3}{6}\left(\frac{\p^3f}{\p x_1^3}(\vec 0)-\epsilon\right).
}
It is clear that $u=0$
and $\frac{\p u}{\p x_1}=\cot\theta_0$ on $\{x_1=0\}$.
Since $0=u(0,\tilde x_2,\ldots,\tilde x_n)$, as computed in \eqref{eq-W_f}, we have
\eq{
    \tilde W_u^2(0,\tilde x_2\ldots,\tilde x_n)
    =&1+\left(\frac{\p u}{\p x_1}(0,\tilde x_2,\ldots,\tilde x_n)\right)^2=\frac{1}{\sin^2\theta_0},
}
and it is easy to see that on $\{x_1=0\}$,
\eq{
    \frac{\p u}{\p x_1}(0,\tilde x_2,\ldots,\tilde x_n)
    =\cos\theta_0\tilde W_u(0,\tilde x_2,\ldots,\tilde x_n).
}

These facts imply: 1. $S'$ is an extension of $\Gamma$; 2. $S'$ meets $\p N^\ast$ with constant contact angle $\theta_0$, due to \eqref{eq-Lem1-2}.
By \cite[Claim 1]{LZ21}, we know that all the partial derivatives (with respect to the coordinates $x_1,\tilde x_2,\ldots,\tilde x_n$) of $u$ and $\tilde f$ agree up to second-order at $\vec 0$, and for sufficiently small $\epsilon$, $\tilde f\geq u$ everywhere in a neighborhood of $p$ with equality holds only at the origin; that is to say, $S'$ touches $\Om$ from outside up to second-order at $p$.

\noindent{\bf Step 2. Constructing the test vector field $X$, which decreases the first variation of $V$ strictly.}

In {\bf Step 1}, we constructed a hypersurface $S'$, meeting $\p N^\ast$ with constant contact angle $\theta_0$, and hence we can use Lemma \ref{Lem-foliation} to obtain local foliations $\{S'_s\}$ and $\{T_t'\}$. We define smooth functions $s,t$ in a neighborhood of $p$, so that $s(q)$ is the unique $s$ such that $q\in S'_s$.
Recall that $s\geq0$ on $\Omega$.

{\bf Claim. }
$\nabla s=\psi e_{n+1}$ for some smooth function $\psi$ such that $\psi\geq c$ near $p$ for some positive constant $c$.
Here $\{e_1,\ldots,e_{n+1}\}$ is a local orthonormal frame
near $p$, as in Lemma \ref{Lem-ONF}.
\begin{proof}[Proof of {\bf Claim}]
Since $s$ is a constant on each leaf $S'_s$, we have that $\nabla s$ is normal to $S'_s$.
It follows from the definition of $e_{n+1}$ that $\nabla s=\psi_1e_{n+1}$,
where $\psi_1$ is smooth in $U\cap\Om$.

By continuity, we find that $\psi\geq\psi(p):=c$ near $p$ (without loss of generality, we may assume that $c=\frac{1}{2}$, otherwise we substitue $s$ by $\frac{\psi(p)}{2}s$).
\end{proof}
Now we define the test vector field $X$ on $N^\ast$ near $p$ by
\eq{\label{def-X}
    X(q)=\phi\left(s(q)\right)(-e_{n+1}(q)),
}
where $\phi(s)$ is the cut-off function defined by
\eq{\label{def-phi}
    \phi(s)=
    \begin{cases}
    {\rm exp}(\frac{1}{s-\epsilon}),
    \quad&0\leq s<\epsilon,\\
    0,&s\geq\epsilon.
    \end{cases}
}
Note that the construction of $X$ is sufficient for our purpose due to Remark \ref{Rem-1st-variation}.
A direction computation then gives, for $0\leq s<\epsilon$, it holds
\eq{
    \frac{\phi'(s)}{\phi(s)}=-\frac{1}{(s-\epsilon)^2}\leq\frac{-1}{\epsilon^2},
}
and hence for any $s\geq0$, we have
\eq{\label{ineq-phi'}
    \phi'(s)\leq-\frac{\phi(s)}{\epsilon^2}.
}
Since $S'$ touches $\Omega$ from outside, we have $s\geq0$ on $\Omega$, and ${\rm spt}(\phi)\cap \Om$ will be close to $p$ as long as $\epsilon$ is small. Thus, if we choose $\epsilon$ to be small enough, then in $\Omega$, our test vector field $X$ will have compact support near $p$. 
Moreover, since $-e_{n+1}(q)$ points into $N^\ast$ for all $q\in\p N^\ast$, we have that $X\in\mathfrak{X}(N^\ast)$. This finishes the construction of our test vector field $X$.

\noindent{\bf Step 3. Testing the first variation by $X$.}

At each $q\in\Om$ that is close to $p$, we consider the bilinear form on $T_qN^\ast$ defined by
\eq{
    Q(u,v)=\left<\nabla_uX,v\right>(q).
}
Let $\{e_1,\ldots,e_{n+1}\}$ be a local orthonormal frame near $p$ as in Lemma \ref{Lem-ONF}.
As computed in \cite[(3.1)]{LZ21}, the bilinear form $Q$ \footnote{Notice that our choice of vector field \eqref{def-X} agrees with the one in \cite{LZ21}, up to a different sign.} can be expressed in this frame by the following matrix
\eq{
    Q=
    \begin{bmatrix}
    \phi A_{11}^{S'_s}&-\phi A^{T_t}_{n+1,j}&0\\
    -\phi A^{T_t}_{i,n+1}&\phi A^{S'_s}_{ij}&0\\
    -\phi A^{T_t}_{n+1,n+1}&-\phi\left<\nabla_{e_{n+1}}e_{n+1},e_j\right>&-\phi'\psi_1
    \end{bmatrix}
}
where $i,j=2,\ldots,n$, and $q\in S'_s\cap T_t$.

Using \eqref{ineq-phi'} and the strictly mean convexity of $S$ at $p$, one finds as in \cite[Lemma 3.2, Lemma 3.3]{LZ21}:
for $\epsilon>0$ small enough, there holds: ${\rm tr}_PQ>0$ for all $n$-dimensional subspaces $P\subset T_qN^\ast$.

Since $(V,W)$ is a fine $\theta$-stationary pair and $V$ is supported in $\Om$, $W$ is supported in $T$, we can use $X$ to test the first variation formula
\eqref{defn-finepair-3} to find:
\eq{
    \int_{G_n(N^\ast)}&{\rm tr}_PQ(q)\rd V(q,P)
    =\frac{\cos\theta_0}{\sin\theta_0}\int_{\p N^\ast}\phi(s(q))\left<-e_{n+1},-\tilde\mu\right>(q)\rd\sigma_V(q)\\
    -&\int_{\p N^\ast}\tilde H(q)\phi(s(q))\left<-e_{n+1},\nu_{N^\ast}\right>(q)\rd\vvert V\vvert(q)+\int_{\p N^\ast}\phi(s(q))\left<-e_{n+1},-\nu_{N^\ast}\right>(q)\rd\sigma_V(q).
}
Recall that $S_s'$ meets $\p N^\ast$ with constant contact angle $\theta_0$, so that
at every $q\in\p N^\ast$, there exists a unit vector $-\bar\mu(q)\in T_q\p N^\ast$ such that
\eq{
    -e_{n+1}(q)=\sin\theta_0(-\bar\mu(q))+\cos\theta_0\nu_{N^\ast}(q),
}
which implies
\eq{\label{eq-tr_PQ}
    \int_{G_n(N^\ast)}&{\rm tr}_PQ(q)\rd V(q,P)
    =\cos\theta_0\int_{\p N^\ast}\phi(s(q))\left<\bar\mu,\tilde\mu\right>(q)\rd\sigma_V(q)\\
    -&\cos\theta_0\int_{\p N^\ast}\phi(s(q))\rd\sigma_V(q)-\cos\theta_0\int_{\p N^\ast}\tilde H(q)\phi(s(q))\rd\vvert V\vvert(q).
}

\noindent{\bf Conclusion of the proof.}

Recall that as we choose $\epsilon$ small enough, in $\Omega$ our test vector field $X$ will have compact support close to $p$, and hence we see from {\bf Step 3} that
\eq{\label{ineq-firstvariation}
    \delta V[X]=\int_{G_n(N^\ast)}{\rm div}_PX(x)\rd V(x,P)
    =\int_{G_n(N^\ast)}{\rm tr}_PQ(q)\rd V(q,P)
    >0,
}
since for any $n$-dimensional affine subspace $P\in T_q N^\ast$, ${\rm div}_PX(q)={\rm tr}_PQ(q)>0$.

Let us check the sign of the RHS of \eqref{eq-tr_PQ}.
By virtue of the condition $\theta_0\in(0,\pi/2)$, the fact that $\phi(s(q))\geq0$ on $\{s\geq0\}$, and the fact that $\left<\bar\mu(q),\tilde\mu(q)\right>\leq1$ locally near $p$, we immediately deduce
\eq{
\int_{G_n(N^\ast)}{\rm tr}_PQ(q)\rd V(q,P)
\leq-\cos\theta_0\int_{\p N^\ast}\tilde H(q)\varphi(s(q))\rd\vvert V\vvert(q).
}
For both case {\bf i} and {\bf ii} of Theorem \ref{Thm-MP-FixedContactAngle}, we readily see that $\int_{\p N^\ast}\tilde H(q)\varphi(s(q))\rd\vvert V\vvert(q)\geq0$, and thanks to $\theta_0\in(0,\pi/2)$ again,
\eq{
    \int_{G_n(N^\ast)}{\rm tr}_PQ(q)\rd V(q,P)\leq0.
}
However, this contradicts to \eqref{ineq-firstvariation} and completes the proof.
\end{proof}








\appendix
\section{Examples of Stationary Pairs with Contact Angle Condition}
In the appendix, we adopt the following convention: for a smooth curve $C$ in $\mfR^2$ and a positive constant $\alpha>0$, we denote by ${\rm var}(C,\alpha)$ the multiplicity $\alpha$, 1-varifold induced by $C$.
The shorthand ${\rm var}(C)$ is used when $\alpha=1$.

As mentioned before, the following example is provided by Gaoming Wang, we record it here for readers' interest.
\begin{example}\label{Exam-A1}
Let $N^\ast$ be an open subset of an upper half-plane in $\mfR^2$ and let $S$ be a mean convex curve in $N^\ast$ to form a domain $\Om$ such that $S$ has contact angle $\theta_0=\frac{2\pi}{3}$. Let $p=S\cap\p N^\ast$, $W={\rm var}(\p N^\ast\cap\p\Om)$.
Choose a ray $R$ such that it has contact angle $\varphi_0$ satisfying $\cos\varphi_0=-\frac14$ and let $V={\rm var}(R\cap\Om,2)$.
Then $(V,W)$ is a $\varphi_0$-stationary pair supported in $\Om$ while the barrier $S$ has contact angle $\theta_0>\varphi_0$.
See Fig. \ref{Fig-example-1}.
\end{example}
\begin{figure}[H]
	\centering
	\includegraphics[height=5cm,width=9cm]{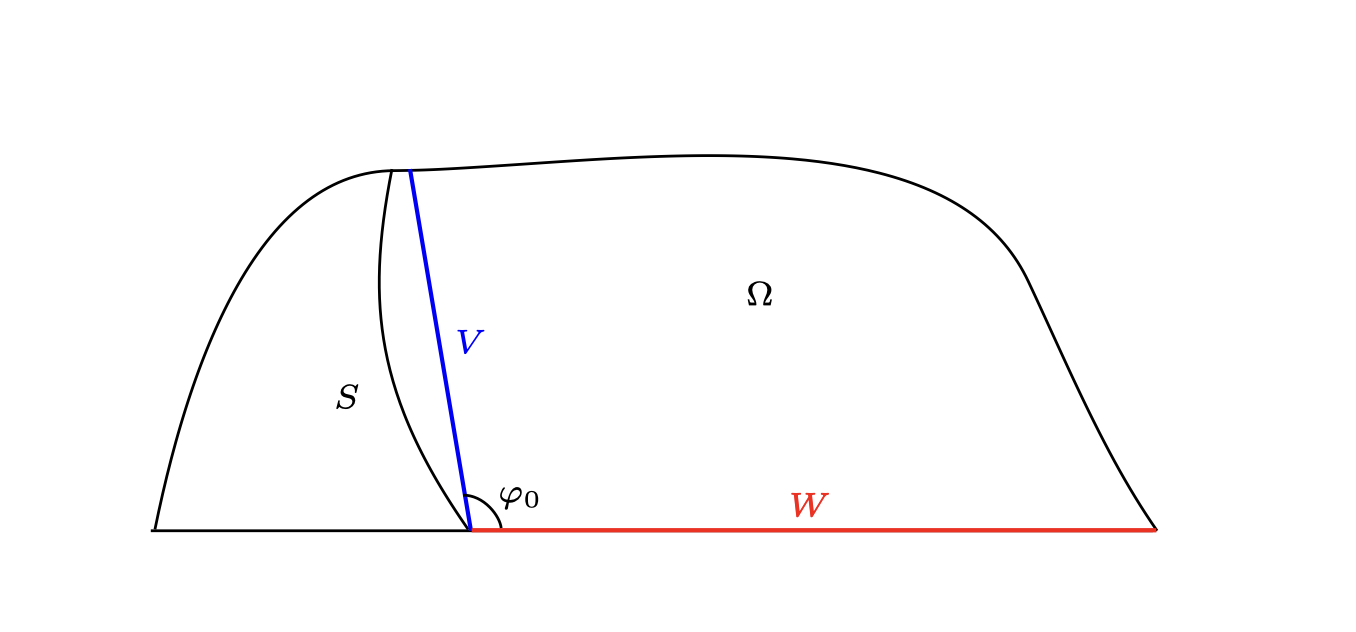}
	\caption{Example A.1.}
	\label{Fig-example-1}
\end{figure}

We end this appendix by introducing a way to construct fine stationary pairs of varifolds $(V,W)$, the case that $\vvert V\vvert(\p N^\ast)>0$ is included.
We may construct as many as possible examples if we like, but the best situation (in the sense that the barrier has the same contact angle as the fine stationary pair does) seems to be obtained when $\vvert V\vvert(\p N^\ast)=0$.
\begin{example}
Let $N^\ast$ be the unit ball in the plane $\mfR^2$ centered at the origin, $\p N^\ast$ is then the unit sphere.
Let $p_1,p_2$ be two points on $\p N^\ast$ such that the line segment joining them, say $L$, has contact angle $\theta_0\in(0,\pi)\setminus\{\frac{\pi}{2}\}$ with $\p N^\ast$, and the enclosed domain is denoted by $\Om$ (see Fig. \ref{Fig-example-2}).
For $\alpha,\beta\geq0$ to be specified latter, define $V={\rm var}(L)+{\rm var}(\p N^\ast\cap\p\Om,\alpha), W={\rm var}(\p N^\ast\cap\p\Om,\beta)$.

Then for $0<\gamma\leq\theta_0$ and $\gamma\neq\frac{\pi}{2}$, the relations
\begin{equation}
\label{condi-gamma-theta}
\begin{split}
        \alpha=&-\cos\theta_0+\sin\theta_0\frac{\cos\gamma}{\sin\gamma},\\
        \beta=&\frac{\sin\theta_0}{\sin\gamma}
\end{split}
\end{equation}
define a fine $\gamma$-stationary pair $(V,W)$.
Moreover, as $\gamma=\theta_0$, one must have $\alpha=0$, $\beta=1$. 
\end{example}

\begin{figure}[H]
	\centering
	\includegraphics[height=8cm,width=15cm]{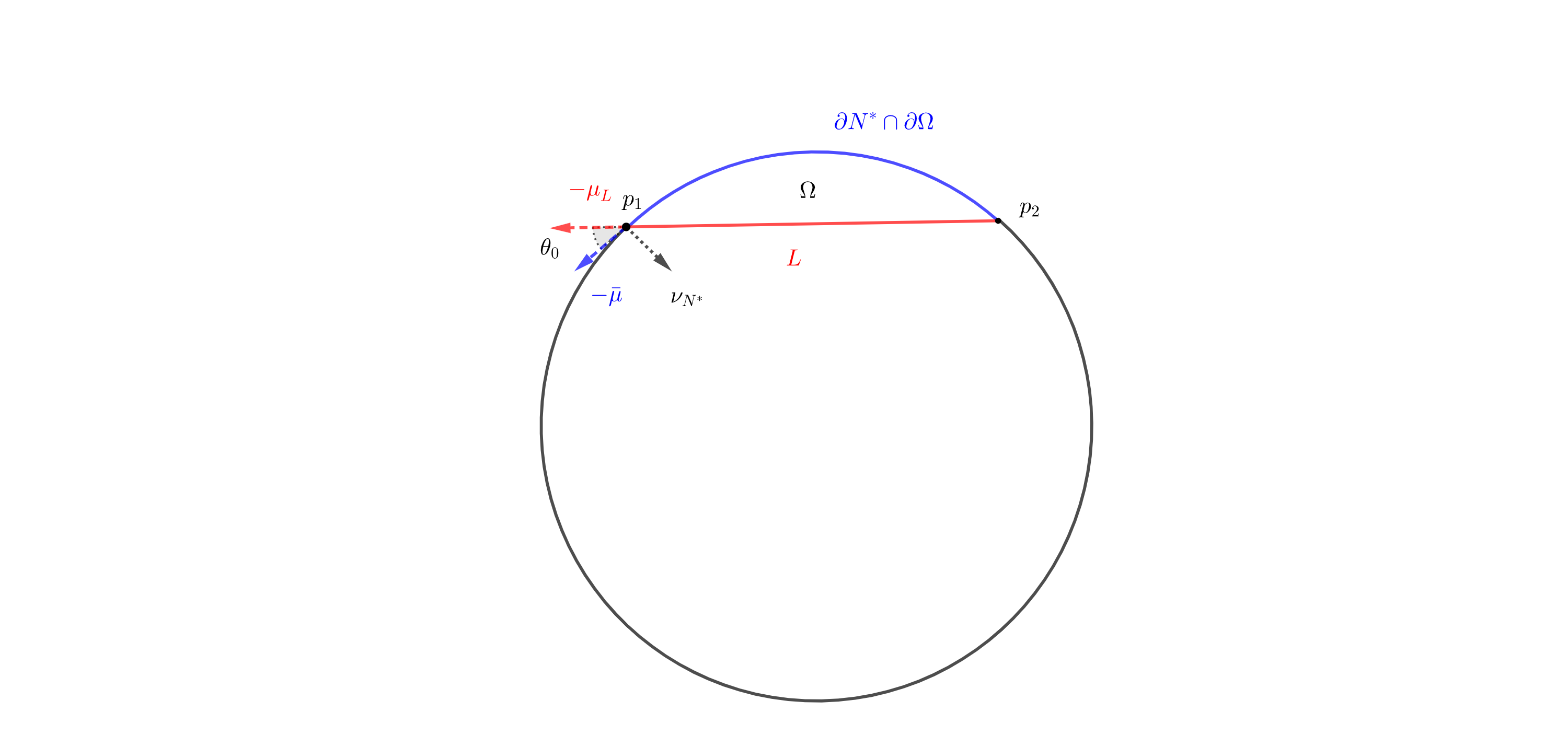}
	\caption{Example A.2.}
	\label{Fig-example-2}
\end{figure}

\begin{proof}
For any $X\in\mathfrak{X}(N^\ast)$, a direct computation shows
\eq{
    \delta V[X]
    &=\int_L{\rm div}_LX(x)\rd\mcH^1(x)+\alpha\int_{\p N^\ast\cap\p\Om}{\rm div}_{\p N^\ast}X(x)\rd\mcH^1(x)\\
    &=\sum_{i=1}^2\left<X(p_i),-\mu_L(p_i)\right>+\alpha\sum_{i=1}^2\left<X(p_i),-\bar\mu(p_i)\right>-\alpha\int_{\p N^\ast\cap\p\Om}\left<X(x),\tilde \mfH(x)\right>\rd\mcH^1(x)\\
    =&(\alpha+\cos\theta_0)\sum_{i=1}^2\left<X(p_i),-\bar\mu(p_i)\right>+\sin\theta_0\sum_{i=1}^2\left<X(p_i),-\nu_{N^\ast}(p_i)\right>-\int_{\p N^\ast}\left<X(x),\tilde \mfH(x)\right>\rd\vvert V\vvert(x),
}
where we have used the fact that $-\mu_L(p_i)=\cos\theta_0(-\bar\mu(p_i))+\sin\theta_0(-\nu_{N^\ast}(p_i))$.

Similarly, for the same $X$, one has
\eq{
    \delta W[X]
    =\beta\int_{\p N^\ast\cap\p\Om}{\rm div}_{\p N^\ast}X(x)\rd\mcH^1(x)
    =\beta\sum_{i=1}^2\left<X(p_i),-\bar\mu(p_i)\right>-\int_{\p N^\ast}\left<X(x),\tilde \mfH(x)\right>\rd\vvert W\vvert(x),
}
it follows that
\eq{
    &\delta V[X]-\cos\gamma\delta W[X]\\
    =&(\alpha+\cos\theta_0-\beta\cos\gamma)\sum_{i=1}^2\left<X(p_i),-\bar\mu(p_i)\right>+\sin\theta_0\sum_{i=1}^2\left<X(p_i),-\nu_{N^\ast}(p_i)\right>\\
    &-\int_{\p N^\ast}\left<X(x),\tilde H(x)\right>\rd\left(\vvert V\vvert-\cos\gamma\vvert W\vvert\right)(x),
}
thus to make $(V,W)$ a $\gamma$-stationary pair, we have to require that
\eq{\label{condi-alpha-beta}
    \alpha+\cos\theta_0-\beta\cos\gamma=0,
}
in which case (compared to Proposition \ref{Prop-MP21-3-1}) $\sigma_V$ is given by $\sin\theta_0\mcH^0\llcorner\{p_1,p_2\}$.
As shown above, for any $X\in\mathfrak{X}_t(N^\ast)$, we have
\eq{
    \delta W[X]=\beta\sum_{i=1}^2\left<X(p_i),-\bar\mu(p_i)\right>,
}
and hence to make $(V,W)$ a fine $\gamma$-stationary pair, we have to
further require that
\eq{
    \beta=\frac{\sin\theta_0}{\sin\gamma}.
}
This, together with \eqref{condi-alpha-beta}, yields \eqref{condi-gamma-theta}.

 From \eqref{condi-gamma-theta}, we observe that $\alpha\geq0$ is in fact equivalent to $\sin(\theta_0-\gamma)\geq0$, that is, $0<\gamma\leq\theta_0<\pi$, and the equality occurs if and only if $\alpha=0,\beta=1$, which completes the proof.
\end{proof}
As a specific choice, if we choose $\theta_0=\frac{\pi}{3},\gamma=\frac{\pi}{6}, \alpha=1, \beta=\sqrt{3}$, then we obtain a  
fine $\frac{\pi}{6}$-stationary pair $(V,W)$, with $\vvert V\vvert(\p N^\ast)>0$.

\bibliographystyle{alpha}
\bibliography{BibTemplate.bib}

\begin{thebibliography}{GLZ20}

\bibitem[All72]{Allard72}
William~K. Allard.
\newblock On the first variation of a varifold.
\newblock {\em Ann. Math. (2)}, 95:417--491, 1972.
\newblock doi: \href{https://doi.org/10.2307/1970868}{10.2307/1970868}.

\bibitem[CM11]{CM11}
Tobias~Holck Colding and William P.~II Minicozzi.
\newblock {\em A course in minimal surfaces}, volume 121 of {\em Grad. Stud.
  Math.}
\newblock Providence, RI: American Mathematical Society (AMS), 2011.

\bibitem[DD21]{MP21}
Luigi {De Masi} and Guido {De Philippis}.
\newblock Min-max construction of minimal surfaces with a fixed angle at the
  boundary, 2021.
\newblock \href{https://arxiv.org/abs/2111.09913}{arXiv:2111.09913}.

\bibitem[DM21]{DeMasi21}
Luigi De~Masi.
\newblock Rectifiability of the free boundary for varifolds.
\newblock {\em Indiana Univ. Math. J.}, 70(6):2603--2651, 2021.
\newblock doi:
  \href{https://doi.org/10.1512/iumj.2021.70.9401}{10.1512/iumj.2021.70.9401}.

\bibitem[DM22]{DeMasi22}
Luigi De~Masi.
\newblock Existence and properties of minimal surfaces and varifolds with
  contact angle conditions, 2022.
\newblock Thesis (Ph.D.)--SISSA,
  \href{http://hdl.handle.net/20.500.11767/129590}{url}.

\bibitem[GLZ20]{GLZ20}
Qiang Guang, Martin Man-chun Li, and Xin Zhou.
\newblock Curvature estimates for stable free boundary minimal hypersurfaces.
\newblock {\em J. Reine Angew. Math.}, 759:245--264, 2020.
\newblock doi:
  \href{https://doi.org/10.1515/crelle-2018-0008}{10.1515/crelle-2018-0008}.

\bibitem[GT01]{GT01}
David Gilbarg and Neil~S. Trudinger.
\newblock {\em Elliptic partial differential equations of second order}.
\newblock Class. Math. Berlin: Springer, reprint of the 1998 ed. edition, 2001.

\bibitem[JT03]{JT03}
Luqu{\'e}sio~P. Jorge and Friedrich Tomi.
\newblock The barrier principle for minimal submanifolds of arbitrary
  codimension.
\newblock {\em Ann. Global Anal. Geom.}, 24(3):261--267, 2003.
\newblock doi:
  \href{https://doi.org/10.1023/A:1024791501324}{10.1023/A:1024791501324}.

\bibitem[KT17]{KT17}
Takashi Kagaya and Yoshihiro Tonegawa.
\newblock A fixed contact angle condition for varifolds.
\newblock {\em Hiroshima Math. J.}, 47(2):139--153, 2017.
\newblock doi:
  \href{https://doi.org/10.32917/hmj/1499392823}{10.32917/hmj/1499392823}.

\bibitem[LZ21a]{LZ21}
Martin Man-chun Li and Xin Zhou.
\newblock A maximum principle for free boundary minimal varieties of arbitrary
  codimension.
\newblock {\em Commun. Anal. Geom.}, 29(6):1509--1521, 2021.
\newblock doi:
  \href{https://doi.org/10.4310/CAG.2021.v29.n6.a7}{10.4310/CAG.2021.v29.n6.a7}.

\bibitem[LZ21b]{LZ21-a}
Martin Man-Chun Li and Xin Zhou.
\newblock Min-max theory for free boundary minimal hypersurfaces. {I}:
  {Regularity} theory.
\newblock {\em J. Differ. Geom.}, 118(3):487--553, 2021.
\newblock doi:
  \href{https://doi.org/10.4310/jdg/1625860624}{10.4310/jdg/1625860624}.

\bibitem[LZZ21]{LZZ21}
Chao {Li}, Xin {Zhou}, and Jonathan~J. {Zhu}.
\newblock Min-max theory for capillary surfaces, 2021.
\newblock \href{https://arxiv.org/abs/2111.09924}{arXiv:2111.0992}.

\bibitem[Pit81]{Pitts81}
Jon~T. Pitts.
\newblock {\em Existence and regularity of minimal surfaces on {Riemannian}
  manifolds}, volume~27 of {\em Math. Notes (Princeton)}.
\newblock Princeton University Press, Princeton, NJ, 1981.

\bibitem[Sim83]{Sim83}
Leon Simon.
\newblock {\em Lectures on geometric measure theory}, volume~3 of {\em Proc.
  Cent. Math. Anal. Aust. Natl. Univ.}
\newblock Australian National University, Centre for Mathematical Analysis,
  Canberra, 1983.

\bibitem[Whi10]{White10}
Brian White.
\newblock The maximum principle for minimal varieties of arbitrary codimension.
\newblock {\em Commun. Anal. Geom.}, 18(3):421--432, 2010.
\newblock doi:
  \href{https://doi.org/10.4310/CAG.2010.v18.n3.a1}{10.4310/CAG.2010.v18.n3.a1}.

\bibitem[XZ23]{XZ23}
Chao Xia and Xuwen Zhang.
\newblock Alexandrov-type theorem for singular capillary cmc hypersurfaces in
  the half-space, 2023.
\newblock \href{https://arxiv.org/abs/2304.01735}{arXiv:2304.01735}.

\bibitem[Zha22]{Zhang22}
Xuwen Zhang.
\newblock A maximum principle for codimension-1 stationary varifolds under
  fixed contact angle condition, 2022.
\newblock \href{https://arxiv.org/abs/2205.07643v4}{arXiv:2205.07643v4}.

\end{thebibliography}


\end{document}